\documentclass[11pt]{article}
\usepackage{amssymb,amsmath}
\usepackage[mathscr]{eucal}
\usepackage{graphicx}
\usepackage[cm]{fullpage}
\usepackage[english]{babel}
\usepackage[latin1]{inputenc}
\usepackage{amsthm} 
\usepackage{stmaryrd} 

\font\petite=cmmi10 at 8pt  

\def\dom{\mathop{\mathrm{Dom}}\nolimits}
\def\im{\mathop{\mathrm{Im}}\nolimits}
\def\reg{\mathop{\mathrm{Reg}}\nolimits}
\def\pv#1{\ensuremath{{\sf#1}}}
\def\pvid#1{{\llbracket#1\rrbracket}}
\def\malcev{\mathbin{\hbox{$\bigcirc$\rlap{\kern-9pt\raise0,75pt\hbox{\petite m}}}}}

\newcommand{\I}{\hbox{$\mathcal I$}}
\def\relR{\mathop{\mathscr R}}
\def\relL{\mathop{\mathscr L}}
\def\relD{\mathop{\mathscr D}}
\def\relJ{\mathop{\mathscr J}}
\def\relI{\mathop{\mathscr I}}
\newcommand{\inv}[1]{{#1}^{-1}} 

\newtheorem{theorem}{Theorem}
\newtheorem{proposition}[theorem]{Proposition}
\newtheorem{lemma}[theorem]{Lemma}

\title{The Vagner-Preston representation of a block-group}
\author{V.H. Fernandes\footnote{
This work is funded by national funds through the FCT - Funda\c c\~ao para a Ci\^encia e a Tecnologia, I.P., under the scope of the project UIDB/00297/2020 (Center for Mathematics and Applications).}}

\newcommand{\lastpage}{\addresss}

\newcommand{\addresss}{\small \sf  
\noindent{\sc V\'\i tor H. Fernandes}, 
CMA, Departamento de Matem\'atica, 
Faculdade de Ci\^encias e Tecnologia, 
Universidade NOVA de Lisboa, 
Monte da Caparica, 
2829-516 Caparica, 
Portugal; 
e-mail: vhf@fct.unl.pt. 
}

\begin{document}

\maketitle \vspace*{-1.5cm}

\renewcommand{\thefootnote}{}



\renewcommand{\thefootnote}{\arabic{footnote}}
\setcounter{footnote}{0}

\begin{abstract}
In this short note we construct an extension of 
the Vagner-Preston representation for block-groups and show that its kernel is the largest congruence that separates regular elements.
\end{abstract}

\subsection*{Introduction} 

A well-known result due to Vagner and Preston states that every inverse
semigroup $S$ has a faithful representation into $\I(S)$, the symmetric inverse semigroup on $S$. In fact, the mapping 
$$
\begin{array}{rccccccc}
\phi: & S & \rightarrow & \I(S)&&&\\
        & s & \mapsto & \phi_s:&Sss^{-1}&\rightarrow&Ss^{-1}s\\
         &&&                    & x & \mapsto &xs
\end{array}
$$
is an injective homomorphism of semigroups. 
Another classical representation of an inverse semigroup is the {\it Munn
representation}: for every inverse semigroup $S$,
denoting by $E$ the semilattice of all the idempotents of $S$, 
a homomorphism from $S$ into $\I(E)$ is defined by 
$$
\begin{array}{rccccccc}
\delta: & S & \rightarrow & \I(E)&&&\\
        & s & \mapsto & \delta_s:&Ess^{-1}&\rightarrow&Es^{-1}s\\
         &&&                    & e & \mapsto &s^{-1}es.
\end{array}
$$
Notice that, in general,
the Munn representation of an inverse semigroup is not injective.
Indeed, the kernel of $\delta$ is the largest 
idempotent-separating congruence on $S$. Therefore, $\delta$ is an
injective homomorphism if and only if $S$ is a fundamental semigroup
(see \cite{Howie:1995} or \cite{Petrich:1984}, for more details).

In this work, for finite semigroups, our goal is to extend the Vagner-Preston representation to semigroups 
whose elements have at most one inverse.  
The finite semigroups with this property form the 
well known class $\pv{BG}$ of semigroups called {\it block-groups}. 
This class constitutes a pseudovariety of semigroups 
(i.e. a class of 
finite semigroups closed under homomorphic images of 
subsemigroups and finitary direct products) 
and plays a main role in 
the following celebrated result: 
$$\diamondsuit\pv{G}=\pv{PG}=\pv{J*G}=\pv{J\malcev G}=\pv{BG}=\pv{EJ},
$$
where $*$ and $\malcev$ denote respectively the semidirect product and the Mal'cev product of pseudovarieties, 
$\pv{G}$ and $\pv{J}$ denote the pseudovarieties of all  
groups and of all $\relJ$-trivial semigroups, respectively, 
$\pv{PG}$ and $\diamondsuit\pv{G}$ denote the pseudovarieties 
generated by all power monoids of groups and by all 
Sch\"utzenberger products of groups, respectively,   
and, finally, $\pv{EJ}$ denotes the pseudovariety of all 
semigroups whose idempotents generate a $\relJ$-trivial semigroup. 
See \cite{Pin:1995} for precise definitions and 
for a complete story of these equalities.  
In \cite{Fernandes:2008SF} we extended the Munn representation for block-groups and, in \cite{Fernandes:2008}, by using this extension, 
we add two other equalities to the above series by showing that 
$$
\pv{BG}=\pv{EI}\malcev\pv{Ecom}=\pv{N}\malcev\pv{Ecom}, 
$$ 
where $\pv{EI}$, $\pv{N}$ and $\pv{Ecom}$ denote the pseudovarieties of 
all semigroups with just one idempotent, of 
all nilpotent semigroups and of all semigroups with commuting idempotents, respectively. 

\medskip 

We assume some knowledge on semigroups, namely on Green's relations, regular elements and
inverse semigroups. Possible references are \cite{Howie:1995,Petrich:1984}. 
For general background on pseudovarieties, pseudoidentities and other stuff on finite semigroups, 
we refer the reader to Almeida's book \cite{Almeida:1995}. 

\smallskip 

All semigroups considered in this paper are finite. 

\subsection*{Extending the Vagner-Preston representation to block-groups}\label{sec1}

We start now the construction of our Vagner-Preston representation 
for block-groups.  
It is easy to show that a block-group can also be defined as a semigroup 
such that every $\relR$-class and every $\relL$-class has at most 
one idempotent.
Clearly, the class of all block-groups comprises all semigroups with 
commuting idempotent  and, in particular, all inverse semigroups.  

\smallskip 

Let $S$ be a semigroup. We denote by $E(S)$ the set of all idempotents 
of $S$ and by $\reg(S)$ the set of all regular elements of $S$. 
Recall the definition of the quasi-orders
$\leq_{\relR}$ and $\leq_{\relL}$ associated to the Green relations
${\relR}$ and ${\relL}$, respectively: for all $s, t\in S$,
$$
s\leq_{\relR}t \mbox{ if and only if }  sS^1\subseteq tS^1
$$
and
$$
s\leq_{\relL}t \mbox{ if and only if }  S^1s\subseteq S^1t,
$$
where $S^1$ denotes the monoid
obtained from $S$ through the adjoining of an identity if $S$ has
none and denotes $S$ otherwise. 

We associate to each element $s\in S$ the 
following two subsets of $E(S)$: 
$$
\relR(s)=\{e\in E(S)\mid e\le_{\relR}s\}\quad\text{and}\quad  
\relL(s)=\{e\in E(S)\mid e\le_{\relL}s\}.
$$

The next two lemmas establish some properties of these sets. 

\begin{lemma} \label{Munn1}  
Let $S$ be a semigroup and $s\in S$. One has:
\begin{description} 
\item (i) If $e\in\relR(s)$ then $es\in \reg(S)$;
\item (ii) If $e\in\relL(s)$ then $se\in \reg(S)$.
\end{description}
\end{lemma} 
\begin{proof}
Take $s\in S$ and $e\in\relR(s)$. 
Then, $e=sx$, for some $x\in S^1$, 
whence 
$ 
es=eees=e(sx)es=(es)x(es) 
$  
and so $es\in\reg(S)$. 
Thus, we proved (i). Similarly, we can prove (ii). 
\end{proof}

Given a block-group $S$ we denote by $s^{-1}$ the unique 
inverse of a regular element $s\in S$. 

\begin{lemma} \label{Munn2}  
Let $S$ be a block-group and $s\in S$. One has:  
\begin{description}
\item (i) If $e\in\relR(s)$ then 
$e=(es)\inv{(es)}=s\inv{(es)}=s\inv{(es)}e$, 
$\inv{(es)}(es)\in\relL(s)$ and $\inv{(es)}(es)\,{\relD}\,e$;  
\item (ii) If $e\in\relL(s)$ then $e=\inv{(se)}(se)=\inv{(se)}s=e\inv{(se)}s$,
$(se)\inv{(se)}\in\relR(s)$ and $(se)\inv{(se)}\,{\relD}\,e$.
\end{description}
\end{lemma} 
\begin{proof}
We just prove (i). The property (ii) can be proved similarly. 

Let $s\in S$ and $e\in\relR(s)$. Take $x\in S^1$ such that $e=sx$. 
As $es\in\reg(S)$, we can consider the idempotent $es\inv{(es)}$.
Now, since $e=ee=esx=es\inv{(es)}esx$, then 
$e$ and $es\inv{(es)}$ are $\relR$-related idempotents, whence 
$e=es\inv{(es)}$.
On the other hand, we have $es=eees=esxes$. Then $xesx$ is the inverse of 
$es$, whence $\inv{(es)}=xesx$ and so $s\inv{(es)}=sxesx=eee=e$. 
Also, the equality $e=s\inv{(es)}e$ follows immediately.  
Next, it is clear that $\inv{(es)}(es)\in\relL(s)$. Moreover,
since $e=es\inv{(es)}$ and $es=es\inv{(es)}es$,
then $e\,{\relR}\,es\,{\relL}\,\inv{(es)}es$, 
whence $\inv{(es)}(es)\,{\relD}\,e$, as required. 
\end{proof}

\smallskip 

Next, we recall that in \cite{Fernandes:2008SF}, 
for a block-group $S$, we showed that the mapping 
$$
\begin{array}{rccccccc}
\delta: & S & \rightarrow & \I(E(S))&&&\\
        & s & \mapsto & \delta_s:&\relR(s)&\rightarrow&\relL(s)\\
         &&&                    & e & \mapsto &\inv{(es)}(es) 
\end{array}
$$
is an idempotent-separating homomorphism. 
Since for an inverse semigroup $S$ this  
homomorphism coincides with the (usual) Munn representation of $S$, 
then also for a block-group $S$ we called to $\delta$ the {\it Munn representation} of $S$. 
Moreover, such as for inverse semigroups, the kernel of the Munn representation of a block-group $S$ 
is the largest idempotent-separating congruence on $S$. 

\smallskip 

Recall also that 
$$
\pv{BG}=\pvid{(x^\omega y^\omega)^\omega=(y^\omega x^\omega)^\omega} = 
\pvid{(x^\omega y^\omega)^\omega x^\omega=(x^\omega y^\omega)^\omega=y^\omega(x^\omega y^\omega)^\omega}
$$ 
(see \cite{Pin:1995}) and, given any semigroup $S$, 
the natural partial order $\le$ of $E(S)$ is defined by 
$e\le f$ if and only if $e=ef=fe$, for all $e,f\in E(S)$. 
Our next result shows that, 
for a semigroup $S\in\pv{BG}$, one of the last two equalities 
suffices to define the natural partial order on $E(S)$ and, like for inverse semigroups, with this order, $E(S)$ forms a semilattice. 

\begin{proposition}\label{idpBG1}
Let $S$ be a block-group and $e,f\in E(S)$. 
Then, we have 
$$ 
e\le f \quad\text{if and only if}\quad e=ef \quad\text{if and only if}\quad e=fe.
$$ 
Moreover, $(E(S),\le)$ is a meet-semilattice, where 
the infimum $e\wedge f$ of $e$ and $f$ is equal to $(ef)^\omega$.
\end{proposition}
\begin{proof}
Considering the definition of the partial order $\le$, to prove the first part of the statement, it suffices to show that 
$e=ef$ is equivalent to $e=fe$. So, by admitting that $e=ef$, we have $e=(ef)^\omega= f (ef)^\omega=f e$. 
Similarly, one can prove that $e=fe$ implies $e=ef$. 

Next, we  prove the second part. Since $(ef)^\omega e=(ef)^\omega =f(ef)^\omega$, we have $(ef)^\omega\le e, f$. 
On the other hand, let $g\in E(S)$ be such that $g\le e, f$. Then, 
$g=gf=(ge)f=g(ef)$, which imply that $g=g(ef)^\omega$ and so $g\le(ef)^\omega$. 
Hence,  $e\wedge f$ exists and it is equal to $(ef)^\omega$, as required. 
\end{proof}

\smallskip 

Now, for a semigroup $S$ and for each $s\in S$, consider the 
following two subsets of $S$: 
$$
\relD(s)=\bigcup\{Se\mid e\in\relR(s)\}\quad\text{and}\quad  
\relI(s)=\bigcup\{Se\mid e\in\relL(s)\}.
$$
Notice that $x\in\relD(s)$ (respectively, $x\in\relI(s)$) if and only if $x=xe$, for some $e\in\relR(s)$ (respectively, $e\in\relL(s)$). 

Given $s,t\in S$, it is easy to show that $\relI(s)=\relI(t)$ if and only if $\relL(s)=\relL(t)$. In a block-group, 
as a consequence of Proposition \ref{idpBG1}, we also have: 

\begin{lemma}\label{idpBG3}
Let $S$ be a block-group and $s,t\in S$. Then $\relD(s)=\relD(t)$ if and only if $\relR(s)=\relR(t)$. 
\end{lemma}
\begin{proof}
It is clear that $\relR(s)=\relR(t)$ implies $\relD(s)=\relD(t)$. Conversely, let us suppose that $\relD(s)=\relD(t)$ and take $e\in\relR(s)$. 
Then $e\in Se\subseteq\relD(s)=\relD(t)$, whence there exists $f\in\relR(t)$ such that $e\in Sf$. 
So, we have $f=t(ft)^{-1}$ and $e=ef$. Now, by Proposition \ref{idpBG1}, we also have $e=fe$, and so 
$e=fe=t(ft)^{-1}e$, from which follows that $e\in\relR(t)$.  Hence $\relR(s)\subseteq\relR(t)$. 
Similarly, we may show that $\relR(t)\subseteq\relR(s)$ and so $\relR(s)=\relR(t)$. 
\end{proof}

Let $S$ be a block-group, $s\in S$ and $e\in\relR(s)$. If $x\in Se$ then $x=xe$ and $xs=xes=x(es)(es)^{-1}(es)=(xs)(es)^{-1}(es)$, whence 
$xs\in S(es)^{-1}(es)$. Since $(es)^{-1}(es)\in\relL(s)$, we may deduce that if $x\in\relD(s)$ then $xs\in\relI(s)$. 
Moreover, we have: 

\begin{proposition} \label{VP3} 
Let $S$ be a block-group and $s\in S$. Then, the mapping 
$$
\begin{array}{rccc}
\phi_s: & \relD(s) & \rightarrow & \relI(s)\\
& x & \mapsto & xs
\end{array} 
$$
is a bijection preserving $\relR$-classes.
\end{proposition}
\begin{proof} Let $s\in S$. 

Clearly, $\phi_s$ preserves $\relR$-classes. 
In fact, let $x\in\relD(s)$ and $e\in\relR(s)$ be such that $x=xe$. 
Then $xs=xs$ and $x=xe=xs(es)^{-1}$, whence $x\relR xs$. 

In order to prove that $\phi_s$ is injective, let $x,y\in\relD(s)$ be such that $xs=ys$. 
Let $e,f\in\relR(s)$ be such that $x=xe$ and $y=ye$. 
Since $xs=ys$, we have $xs(es)^{-1}=ys(es)^{-1}$ and so $x=xe=ye$. Similarly, we have $y=xf$. 
Consequently, 
$$
x=ye=xfe=\cdots=x(fe)^\omega=xf(fe)^\omega=y(fe)^\omega=y(ef)^\omega=\cdots=yef=xf=y.
$$
Hence, $\phi_s$ is injective. 

Lastly, we prove that $\phi_s$ is surjective. Let $y\in\relI(s)$. Then, there exists $f\in\relL(s)$ such that $y=yf$. 
Take $x=y(sf)^{-1}$ and $e=(sf)(sf)^{-1}$. Then $e\in\relR(s)$ and $xe=y(sf)^{-1}(sf)(sf)^{-1}=y(sf)^{-1}=x$, whence $x\in\relD(s)$. 
Moreover, $x\phi_s=xs=y(sf)^{-1}s=yf=y$. Thus, $\phi_s$ is surjective, as required. 
\end{proof}

Now, we may present our main result. 

\begin{theorem} \label{VP4}  
Let $S$ be a block-group. 
Then, the mapping 
$$
\begin{array}{rccc}
\phi: & S & \rightarrow & \I(S) \\
& s & \mapsto & \phi_s
\end{array} 
$$
is a homomorphism of semigroups such that: 
\begin{enumerate}
\item $\phi$ is injective in $\reg(S)$, i.e. the kernel of $\phi$ separates regular elements;
\item If $s\in\reg(S)$ then $\relD(s)=Sss^{-1}$ and $\relI(s)=Ss^{-1}s$. 
In particular, if $S$ is an inverse semigroup then $\phi$ is the usual Vagner-Preston representation of $S$. 
\end{enumerate}
Furthermore, the kernel of $\phi$ is the largest congruence on $S$ that separates regular elements. 
\end{theorem} 
\begin{proof}
First of all, we prove that $\phi$ is a homomorphism of semigroups.  

Let $s,t\in S$. For $x\in\dom(\phi_s\phi_t)\cap\dom(\phi_{st})$, we have 
$x\phi_s\phi_t=(xs)\phi_t=xst=x\phi_{st}$. Hence, in order to prove that $\phi_s\phi_t=\phi_{st}$, 
it suffices to show that $\dom(\phi_s\phi_t)=\dom(\phi_{st})$. 
With this goal in mind, take $x\in\dom(\phi_s\phi_t)$, i.e. $x\in\dom(\phi_s)=\relD(s)$ and $xs=x\phi_s\in\dom(\phi_t)=\relD(t)$. 
Hence, there exist $e\in\relR(s)$ and $f\in\relR(t)$ such that $x=xe$ and $xs=xsf$. 
Observe that $e=s(es)^{-1}$ and $f=t(ft)^{-1}$. 
Let $g=(st(ft)^{-1}(es)^{-1})^\omega$. 
Then, clearly, $g\in\relR(st)$ and, on the other hand, 
$$
x=xe=xs(es)^{-1}=xsf(es)^{-1}=xst(ft)^{-1}(es)^{-1}=\cdots=x(st(ft)^{-1}(es)^{-1})^\omega=xg,
$$
whence $x\in\relD(st)=\dom(\phi_{st})$. 
Conversely, take $x\in\dom(\phi_{st})=\relD(st)$. 
Let $g\in\relR(st)$ be such that $x=xg$. Then $g=st(gst)^{-1}$. Clearly, we also have $g\in\relR(s)$, whence $x\in\relD(s)=\dom(\phi_s)$. 
On the other hand, take $f=(t(gst)^{-1}s)^\omega$. 
Then, $f\in\relR(t)$ and 
$$
xs=xgs=xst(gst)^{-1}s=\cdots=xs(t(gst)^{-1}s)^\omega=xsf,
$$ 
whence $x\phi_s=xs\in\relD(t)=\dom(\phi_t)$ and so $x\in\dom(\phi_s\phi_t)$. Thus, $\dom(\phi_s\phi_t)=\dom(\phi_{st})$.

Therefore, we showed that $\phi$ is a homomorphism of semigroups.  

Next, we consider an element $s\in\reg(S)$ and prove that $\relD(s)=Sss^{-1}$ and $\relI(s)=Ss^{-1}s$. 

First, we prove that $\relD(s)=Sss^{-1}$. Take $x\in\relD(s)$. Then, there exists $e\in\relR(s)$ such that $x=xe$. 
Then $ss^{-1}e=ss^{-1}s(es)^{-1}=s(es)^{-1}=e$ and so, by Proposition \ref{idpBG1}, we also have $e=ess^{-1}$. 
Hence, $x=xe=xess^{-1}\in Sss^{-1}$ and so $\relD(s)\subseteq Sss^{-1}$. 
Conversely, since $ss^{-1}\in\relR(s)$, by definition, we have $Sss^{-1}\subseteq \relD(s)$, whence $\relD(s)=Sss^{-1}$. 

Now, we show that $\relI(s)=Ss^{-1}s$. Since $s^{-1}s\in\relL(s)$, it follows immediately that $Ss^{-1}s\subseteq \relI(s)$. 
Conversely, take $x\in\relI(s)$. Then, there exists $e\in\relL(s)$ such that $x=xe$. 
Observe that $e=(se)^{-1}s$. 
Then $x=xe=x(se)^{-1}s=x(se)^{-1}ss^{-1}s\in Ss^{-1}s$, 
whence $\relD(s)\subseteq Ss^{-1}s$ and so $\relI(s)=Ss^{-1}s$. 

Our next step is to show that $\phi$ is injective in $\reg(S)$. Let $s,t\in\reg(S)$ be such that $\phi_s=\phi_t$. 
Then, in particular, 
$Sss^{-1}=\dom(\phi_s)=\dom(\phi_t)=Stt^{-1}$, whence $ss^{-1},tt^{-1}\in\dom(\phi_s)=\dom(\phi_t)$ and so 
$s=(ss^{-1})\phi_s=(ss^{-1})\phi_t=ss^{-1}t$ and, similarly, $t=tt^{-1}s$. 
Now, from $s=ss^{-1}t$ we get $ts^{-1}=ts^{-1}ss^{-1}=ts^{-1}ss^{-1}ts^{-1}=ts^{-1}ts^{-1}$ and 
from $t=tt^{-1}s$ we get $ts^{-1}=tt^{-1}ss^{-1}$. Then $tt^{-1}ss^{-1}$ is an idempotent and so $(tt^{-1}ss^{-1})^\omega=tt^{-1}ss^{-1}$. 
Similarly $(ss^{-1}tt^{-1})^\omega=ss^{-1}tt^{-1}$, 
whence 
$$
s=ss^{-1}t=ss^{-1}tt^{-1}t=(ss^{-1}tt^{-1})^\omega t= (tt^{-1}ss^{-1})^\omega t=tt^{-1}ss^{-1}t=tt^{-1}s=t. 
$$

Finally, in order to prove that $\ker(\phi)$ is the largest congruence on $S$ that separates regular elements, let $\rho$ be a congruence on $S$ that separates regular elements. 

Let $s,t\in S$ be such that $s\rho\,t$. 

Let us show that $\phi_s$ and $\phi_t$ have the same domain, i.e. $\relD(s)=\relD(t)$. As observed before, it suffices to show that $\relR(s)=\relR(t)$. 
With this objective in mind, take $e\in\relR(s)$. Then, we have $e=s(es)^{-1}$ and so $s\rho\,t$ implies $e=s(es)^{-1}\rho\,t(es)^{-1}$, whence  
$e\rho\,(t(es)^{-1})^\omega$, from which follows $e=(t(es)^{-1})^\omega$, as $e$ and $(t(es)^{-1})^\omega$ are regular elements, and so $e\in\relR(t)$. 
Hence $\relR(s)\subseteq\relR(t)$. 
Similarly, we may show that $\relR(t)\subseteq\relR(s)$ and so $\relR(s)=\relR(t)$. 

Now, take $x\in\relD(s)$. Then, $x=xe$, for some $e\in\relR(s)$. As $\relR(s)=\relR(t)$, we have $es, et\in\reg(S)$ and, since $s\rho\,t$ implies $es\rho\,et$, we deduce that $es=et$. Hence, $x\phi_s=xs=xes=xet=xt=x\phi_t$, which allows us to conclude that $\phi_s=\phi_t$ and so $(s,t)\in\ker(\phi)$. 

Thus $\rho\subseteq\ker(\phi)$, as required.
\end{proof} 

\smallskip 

Let $S$ be a block-group. Like for inverse semigroups, we call to the homomorphism $\phi:S\rightarrow\I(S)$ defined in Theorem \ref{VP4} the \textit{Vagner-Preston representation} of $S$. 

\smallskip 

Recall that in \cite{Fernandes:2008} we proved that $\pv{BG}=\pv{N}\malcev\pv{Ecom}$. 
As an application of Theorem \ref{VP4}, we may now give a more direct and simpler proof 
of the inclusion $\pv{BG}\subseteq\pv{N}\malcev\pv{Ecom}$ than the one we did in the aforementioned paper. 

First, we recall that, given a pseudovariety of semigroups $\pv V$, 
a semigroup $S$ is called a $\pv V$-{\it extension} of a semigroup $T$ 
if there exists an onto homomorphism $\varphi:S\longrightarrow T$ such that, 
for every idempotent $e$ of $T$, 
the subsemigroup $e\varphi^{-1}$ of $S$ belongs to $\pv V$. 
In addition, being $\pv W$ another pseudovariety of semigroups,  
the {\it Mal'cev product} $\pv V\malcev\pv W$ is the pseudovariety 
of semigroups generated by all $\pv V$-extensions of elements of $\pv W$. 

Now, like the inclusion $\pv{BG}\subseteq\pv{EI}\malcev\pv{Ecom}$ is an immediate consequence of the construction of the Munn representation for block-groups \cite{Fernandes:2008SF} (recall that $\pv{EI}=\pvid{x^\omega =y^\omega}$ and the Munn representation of a block-group is an idempotent-separating homomorphism), in its turn, the inclusion $\pv{BG}\subseteq\pv{N}\malcev\pv{Ecom}$ is an immediate consequence of the construction of the Vagner-Preston representation for block-groups. In fact, suppose that $S\in\pv{BG}$ and let $\phi:S\rightarrow\I(S)$ be the Vagner-Preston representation. 
Take an idempotent $\epsilon\in\im(\phi)$. Then $\epsilon\phi^{-1}$ is a subsemigroup of $S$ with only one regular element and so $\epsilon\phi^{-1}\in\pv{N}=\pvid{x^\omega=0}$ (indeed, it is easy to check that a finite semigroup has a unique regular element if and only if it is nilpotent). 
As $\im(\phi)$ is a subsemigroup of an inverse semigroup, we have $\im(\phi)\in\pv{Ecom}$, whence $S\in\pv{N}\malcev\pv{Ecom}$, as required. 



\lastpage 
\end{document}